\documentclass[reqno, 12pt]{amsart}
\pdfoutput=1
\makeatletter
\let\origsection=\section \def\section{\@ifstar{\origsection*}{\mysection}}
\def\mysection{\@startsection{section}{1}\z@{.7\linespacing\@plus\linespacing}{.5\linespacing}{\normalfont\scshape\centering\S}}
\makeatother

\usepackage{amsmath,amssymb,amsthm}
\usepackage{mathrsfs}
\usepackage{mathabx}\changenotsign
\usepackage{dsfont}

\usepackage{graphicx}
\usepackage{float}

\usepackage{xcolor}
\usepackage{hyperref}
\hypersetup{
	colorlinks,
	linkcolor={red!60!black},
	citecolor={green!60!black},
	urlcolor={blue!60!black}
}

\definecolor{codelightgray}{gray}{0.8}
\definecolor{codeverylightgray}{gray}{0.9}

\usepackage{array,multirow,colortbl,enumerate}
\usepackage{caption} 

\usepackage{graphicx}

\usepackage[open,openlevel=2,atend]{bookmark}

\usepackage[abbrev,msc-links,backrefs]{amsrefs}
\usepackage{doi}

\renewcommand{\PrintDOI}[1]{\doi{#1}}

\usepackage[OT2, T1]{fontenc}
\usepackage{lmodern}
\usepackage[babel]{microtype}
\usepackage[english]{babel}

\DeclareRobustCommand{\rn}[1]{  {\fontencoding{OT2}\selectfont#1}}

\usepackage{geometry}
\numberwithin{equation}{section}
\numberwithin{figure}{section}

\usepackage{enumitem}

\let\polishlcross=\l
\def\l{\ifmmode\ell\else\polishlcross\fi}

\def\paragraph#1{	\noindent\textbf{#1.}\enspace}

\let\sm=\setminus

\makeatletter
\def\moverlay{\mathpalette\mov@rlay}
\def\mov@rlay#1#2{\leavevmode\vtop{   \baselineskip\z@skip \lineskiplimit-\maxdimen
		\ialign{\hfil$\m@th#1##$\hfil\cr#2\crcr}}}
\newcommand{\charfusion}[3][\mathord]{
	#1{\ifx#1\mathop\vphantom{#2}\fi
		\mathpalette\mov@rlay{#2\cr#3}
	}
	\ifx#1\mathop\expandafter\displaylimits\fi}
\makeatother

\newcommand{\dcup}{\charfusion[\mathbin]{\cup}{\cdot}}

\DeclareFontFamily{U}  {MnSymbolC}{}
\DeclareSymbolFont{MnSyC}         {U}  {MnSymbolC}{m}{n}
\DeclareFontShape{U}{MnSymbolC}{m}{n}{
	<-6>  MnSymbolC5
	<6-7>  MnSymbolC6
	<7-8>  MnSymbolC7
	<8-9>  MnSymbolC8
	<9-10> MnSymbolC9
	<10-12> MnSymbolC10
	<12->   MnSymbolC12}{}
\DeclareMathSymbol{\powerset}{\mathord}{MnSyC}{180}

\usepackage{tikz}
\usetikzlibrary{calc,decorations.pathmorphing,decorations.pathreplacing}
\usetikzlibrary{intersections}
\usetikzlibrary {arrows.meta} 
\pgfdeclarelayer{background}
\pgfdeclarelayer{foreground}
\pgfdeclarelayer{front}
\pgfsetlayers{background,main,foreground,front}

\usepackage{subcaption}
\let\epsilon=\varepsilon

\let\rho=\varrho
\let\theta=\vartheta

\def\cf{\mathrm{cf}}

\newcommand{\ccM}{\mathscr{M}}

\theoremstyle{plain}
\newtheorem{thm}{Theorem}[section]

\newtheorem{claim}[thm]{Claim}
\newtheorem{fact}[thm]{Fact}
\newtheorem{cor}[thm]{Corollary}

\theoremstyle{definition}

\usepackage{accents}

\let\lra=\longrightarrow
\let\phi=\varphi
\let\vn=\varnothing
\def\cf{\mathrm{cf}}

\newcommand{\redge}[8]{

		\ifx\relax#5\relax
		\def\qoffs{0pt}
	\else
		\def\qoffs{#5}
	\fi

				\def\rhedge{
		($#1+#4!\qoffs!-90:#2-#4$) -- 
		($#2+#1!\qoffs!-90:#3-#1$) -- 
		($#3+#2!\qoffs!-90:#4-#2$) -- 
		($#4+#3!\qoffs!-90:#1-#3$) -- cycle}

	\coordinate (12) at ($#1!\qoffs!90:#2$);
	\coordinate (14) at ($#1!\qoffs!-90:#4$);
	\coordinate (23) at ($#2!\qoffs!90:#3$);
	\coordinate (21) at ($#2!\qoffs!-90:#1$);
	\coordinate (34) at ($#3!\qoffs!90:#4$);
	\coordinate (32) at ($#3!\qoffs!-90:#2$);
	\coordinate (41) at ($#4!\qoffs!90:#1$);
	\coordinate (43) at ($#4!\qoffs!-90:#3$);
	
	\def\nrhedge{
		(14) let \p1=($(14)-#1$), \p2=($(12)-#1$) in 
			arc[start angle={atan2(\y1,\x1)}, delta angle={atan2(\y2,\x2)-atan2(\y1,\x1)-360*(atan2(\y2,\x2)-atan2(\y1,\x1)>0)}, x radius=\qoffs, y radius=\qoffs] --
		(21) let \p1=($(21)-#2$), \p2=($(23)-#2$) in 
			arc[start angle={atan2(\y1,\x1)}, delta angle={atan2(\y2,\x2)-atan2(\y1,\x1)-360*(atan2(\y2,\x2)-atan2(\y1,\x1)>0)}, x radius=\qoffs, y radius=\qoffs] --
		(32) let \p1=($(32)-#3$), \p2=($(34)-#3$) in 
			arc[start angle={atan2(\y1,\x1)}, delta angle={atan2(\y2,\x2)-atan2(\y1,\x1)-360*(atan2(\y2,\x2)-atan2(\y1,\x1)>0)}, x radius=\qoffs, y radius=\qoffs] --
		(43) let \p1=($(43)-#4$), \p2=($(41)-#4$) in 
			arc[start angle={atan2(\y1,\x1)}, delta angle={atan2(\y2,\x2)-atan2(\y1,\x1)-360*(atan2(\y2,\x2)-atan2(\y1,\x1)>0)}, x radius=\qoffs, y radius=\qoffs] --
		cycle}

		\ifx\relax#6\relax
		\def\rlwidth{1pt}
	\else
		\def\rlwidth{#6}
	\fi
	
		\ifx\relax#8\relax
		\fill \nrhedge;
	\else
		\fill[#8]\nrhedge;
	\fi

		\ifx\relax#7\relax
		\draw[line width=\rlwidth,rounded corners=\qoffs]\nrhedge;
	\else
		\draw[line width=\rlwidth,#7]\nrhedge;
	\fi
	}

\newcommand{\qedge}[7]{

	\ifx\relax#4\relax
		\def\qoffs{0pt}
	\else
		\def\qoffs{#4}
	\fi

	\def\qhedge{
		($#1+#3!\qoffs!-90:#2-#3$) --
		($#2+#1!\qoffs!-90:#3-#1$) --
		($#3+#2!\qoffs!-90:#1-#2$) -- cycle}

	\coordinate (12) at ($#1!\qoffs!90:#2$);
	\coordinate (13) at ($#1!\qoffs!-90:#3$);
	\coordinate (23) at ($#2!\qoffs!90:#3$);
	\coordinate (21) at ($#2!\qoffs!-90:#1$);
	\coordinate (31) at ($#3!\qoffs!90:#1$);
	\coordinate (32) at ($#3!\qoffs!-90:#2$);
	
	\def\nqhedge{
		(13) let \p1=($(13)-#1$), \p2=($(12)-#1$) in
			arc[start angle={atan2(\y1,\x1)}, delta angle={atan2(\y2,\x2)-atan2(\y1,\x1)-360*(atan2(\y2,\x2)-atan2(\y1,\x1)>0)}, x radius=\qoffs, y radius=\qoffs] --
		(21) let \p1=($(21)-#2$), \p2=($(23)-#2$) in
			arc[start angle={atan2(\y1,\x1)}, delta angle={atan2(\y2,\x2)-atan2(\y1,\x1)-360*(atan2(\y2,\x2)-atan2(\y1,\x1)>0)}, x radius=\qoffs, y radius=\qoffs] --
		(32) let \p1=($(32)-#3$), \p2=($(31)-#3$) in
			arc[start angle={atan2(\y1,\x1)}, delta angle={atan2(\y2,\x2)-atan2(\y1,\x1)-360*(atan2(\y2,\x2)-atan2(\y1,\x1)>0)}, x radius=\qoffs, y radius=\qoffs] --
		cycle}

		\ifx\relax#5\relax
		\def\qlwidth{1pt}
	\else
		\def\qlwidth{#5}
	\fi
	
		\ifx\relax#7\relax
		\fill \nqhedge;
	\else
		\fill[#7]\nqhedge;
	\fi

		\ifx\relax#6\relax
		\draw[line width=\qlwidth,rounded corners=\qoffs]\nqhedge;
	\else
		\draw[line width=\qlwidth,#6]\nqhedge;
	\fi
}

\definecolor{amet}{rgb}{0.54,0.0,1.0}
\definecolor{electricviolet}{rgb}{0.56, 0.0, 1.0}
\definecolor{electricindigo}{rgb}{0.44, 0.0, 1.0}

\begin{document}
\title[Obligatory hypergraphs]{Obligatory hypergraphs}
\author[Christian Reiher]{Christian Reiher}
\address{Fachbereich Mathematik, Universit\"at Hamburg, Hamburg, Germany}
\email{christian.reiher@uni-hamburg.de }

\begin{abstract}
	Erd\H{o}s and Hajnal proved that every graph of uncountable chromatic number 
	contains arbitrarily large finite, complete, bipartite graphs. We extend this 
	result to hypergraphs. 
\end{abstract}
	
\maketitle

\section{Introduction}
For an integer $k\ge 2$ a {\it $k$-uniform hypergraph} is a pair $H=(V, E)$ consisting
of a (finite or infinite) set of {\it vertices} $V$ and a 
set $E\subseteq V^{(k)}=\{e\subseteq V\colon |e|=k\}$ of $k$-element subsets 
of~$V$, called the {\it edges} of~$H$. 
The {\it chromatic number} of a hypergraph $H=(V, E)$, denoted by~$\chi(H)$, is the 
least cardinal $\kappa$ for which there exists a colouring $f\colon V\lra\kappa$ such 
that no edge of $H$ is monochromatic with respect to $f$. A finite hypergraph $F$ is
said to be {\it obligatory} if every hypergraph $H$ with $\chi(H)\ge \aleph_1$ has 
a subhypergraph isomorphic to $F$. 
Erd\H{o}s and Hajnal~\cite{EH66}*{Corollary 5.6 and Theorem 7.4} obtained the 
following characterisation of obligatory graphs.  

\begin{thm}[Erd\H{o}s and Hajnal]\label{thm:11}
	A finite graph is obligatory if and only if it is bipartite. \qed
\end{thm}

For general hypergraphs no comparable result is known, and no plausible conjecture 
has ever been proposed. The earliest contribution to this area we are aware 
of is due to Erd\H{o}s, Hajnal, and Rothschild~\cite{EHRoth}, who showed that 
hypergraphs with two edges intersecting in two or more vertices are non-obligatory.
A systemstic study of the $3$-uniform case was initiated by
Erd\H{o}s, Galvin, and Hajnal~\cite{EGH}, and later continued by  
Komj\'ath~\cites{kom01, kom08, HK08}. Many results from~\cite{kom01} can be 
extended in a straightforward manner to $k$-uniform hypergraphs. For instance, 
every obligatory $k$-uniform hypergraph is $k$-partite, and the class of obligatory, 
$k$-uniform hypergraphs is closed under taking disjoint unions and one-point 
amalgamations. In particular, every finite forest is obligatory, i.e., every finite 
hypergraph $F$ whose set of edges admits an enumeration $E(F)=\{e_0, \dots, e_{n-1}\}$
such that $\big|e_i\cap\bigcup_{j<i}e_j\big|\le 1$ holds for every positive $i<n$.
To the best of our knowledge, however, it has never been shown that for any~$k\ge 3$
there exists an obligatory, $k$-uniform hypergraph which fails to be a forest. 

The goal of this article is to describe a natural class of obligatory hypergraphs
that generalise bipartite graphs. Given an arbitrary graph $F$ and an integer $k$
the {\it expansion} $F^{(k)}$ is defined to be the $k$-uniform hypergraph derived 
from $F$ by adding $k-2$ new vertices to every edge; thus it has $|V(F)|+(k-2)|E(F)|$
vertices and there is a bijection $\varphi\colon E(F)\lra E(F^{(k)})$ such that 
$e\subseteq \phi(e)$ and $e\cap e'=\phi(e)\cap\phi(e')$ hold for all distinct 
$e, e'\in E(F)$. Hypergraph expansions were introduced about twenty years ago by 
Mubayi in his work on the hypergraph Tur\'an problem~\cite{Mubayi}. 
A few examples of the construction are displayed in Figure~\ref{fig11}.

\begin{figure}[h!]
	\begin{subfigure}[b]{.32\textwidth}
		\centering
	\begin{tikzpicture}
		\def\b{.8};
		\coordinate (a12) at (0,0);
		\coordinate (a11) at (-2*\b,0);
		\coordinate (a13) at (2*\b,0);
		\coordinate (a21) at (-\b, 1.732*\b);
		\coordinate (a22) at (\b, 1.732*\b);
		\coordinate (a31) at (0, 3.464*\b);
		
	\qedge{(a11)}{(a21)}{(a12)}{5pt}{.6pt}{red}{red, opacity=.3};
		\qedge{(a12)}{(a22)}{(a13)}{5pt}{.6pt}{red}{red, opacity=.3};
			\qedge{(a21)}{(a31)}{(a22)}{5pt}{.6pt}{red}{red, opacity=.3};
		
		\foreach \i in {11, 12, 13, 21, 22, 31} \fill (a\i) circle (1.5pt);
		
	\end{tikzpicture}
\caption{$C_3^{(3)}$}\label{fig11A}
	\end{subfigure}
\begin{subfigure}[b]{.25\textwidth}
	\centering
\begin{tikzpicture}
		\def\b{1.12}
	\foreach \i in {1,...,4}
	{	\foreach \j in {1,...,4} {\coordinate (a\i\j) at (\j*\b,\i*\b);}}
		
		\redge{(a23)}{(a13)}{(a12)}{(a22)}{5pt}{.6pt}{yellow!80!black}{yellow, opacity=.4};
		\redge{(a43)}{(a33)}{(a32)}{(a42)}{5pt}{.6pt}{yellow!80!black}{yellow, opacity=.4};
		\redge{(a32)}{(a22)}{(a21)}{(a31)}{5pt}{.6pt}{yellow!80!black}{yellow, opacity=.4};
		\redge{(a34)}{(a24)}{(a23)}{(a33)}{5pt}{.6pt}{yellow!80!black}{yellow, opacity=.4};

		\foreach \i in {12, 13, 21, 22, 23, 24, 31, 32, 33, 34, 42, 43} \fill (a\i) circle (1.5pt);
\end{tikzpicture}
\caption{$C_4^{(4)}$}\label{fig11B}
\end{subfigure}
\begin{subfigure}[b]{.32\textwidth}
	\centering
	\begin{tikzpicture}
		\def\h{1.2};
		\def\w{.4};

			\coordinate (a12) at (-3*\w, 0);
				\coordinate (a13) at (-\w, 0);
				\coordinate (a14) at (\w, 0);
					\coordinate (a15) at (3*\w, 0);
				
			\coordinate (a32) at (0,2*\h);		
			\coordinate (a21) at ($(a12)!.5!(a32)$);
			\coordinate (a22) at ($(a15)!.5!(a32)$);
			\coordinate (a33) at  ($(a13)!2!(a22)$);
			\coordinate (a31) at  ($(a14)!2!(a21)$);
			\coordinate (a11) at  ($(a33)!2!(a21)$);
			\coordinate (a16) at  ($(a31)!2!(a22)$);
						
			\qedge{(a12)}{(a21)}{(a32)}{5pt}{.6pt}{red}{red, opacity=.3};
			\qedge{(a15)}{(a22)}{(a32)}{5pt}{.6pt}{red}{red, opacity=.3};
			\qedge{(a13)}{(a22)}{(a33)}{5pt}{.6pt}{red}{red, opacity=.3};
			\qedge{(a14)}{(a21)}{(a31)}{5pt}{.6pt}{red}{red, opacity=.3};
			\qedge{(a33)}{(a21)}{(a11)}{5pt}{.6pt}{red}{red, opacity=.3};
			\qedge{(a31)}{(a22)}{(a16)}{5pt}{.6pt}{red}{red, opacity=.3};

		\foreach \i in {11, 12, 13, 14, 15,16,32,21,22, 33,31} \fill (a\i) circle (1.5pt);				
	\end{tikzpicture}
\caption{$K^{(3)}_{2,3}$}\label{fig11C}
\end{subfigure}
\caption{Some expansions}\label{fig11}
\end{figure}
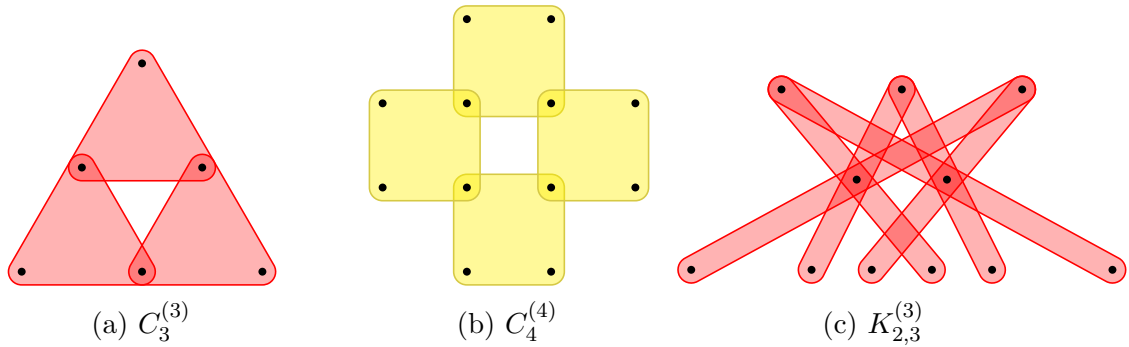

In the result that follows, $K_{n, n}$ refers to the complete bipartite graph
with $n$ vertices in each of its vertex classes, and $K^{(k)}_{n, n}$ denotes 
its $k$-uniform expansion. 
     
\begin{thm}\label{thm:12}
	For all integers $k\ge 2$ and $n\ge 1$ the $k$-uniform hypergraph $K^{(k)}_{n, n}$
	is obligatory. 
\end{thm} 

In particular, for every even $\ell\ge 4$ the expanded $\ell$-cycle $C^{(k)}_\ell$
is obligatory. The hypergraph~$C^{(3)}_3$, on the other hand, was shown to be  
non-obligatory by Erd\H{o}s, Galvin, and Hajnal~\cite{EGH}*{Theorem 11.6}. 
A perhaps gentler introduction to obligatory hypergraphs can be found in 
our recent survey~\cite{girth-survey}*{\S3}.
 
\section{The proof}

We shall exploit the following submultiplicativity property of the chromatic number,
which was observed for finite graphs by Zykov~\cite{Zykov}*{\rn{Teorema}~2}. 
His product argument establishes the general case as well.
 
\begin{fact}\label{f:zykov}
	Let $H=(V, E)$ be a hypergraph. If $E=\bigcup_{i\in I}E_i$, then 
		\[
		\pushQED{\qed} 
		\chi(H)\le \prod_{i\in I} \chi(V, E_i)\,. \qedhere
		\popQED
	\]
	\end{fact}

Often we use the following consequence.

\begin{cor}\label{cor:zykov}
	If $H=(V, E)$ denotes a hypergraph of uncountable chromatic number and 
	$E'\subseteq E$ satisfies $\chi(V, E')\le \aleph_0$, then 
	$\chi(V, E\sm E')\ge \aleph_1$. \qed
\end{cor}
 
\begin{proof}[Proof of Theorem~\ref{thm:12}]
	Arguing indirectly, we fix the smallest integer $k$ such that for some positive 
	integer $n$ the expansion $K^{(k)}_{n, n}$ is non-obligatory. Theorem~\ref{thm:11} 
	tells us that $k\ge 3$. Let~$\kappa$ be the least 
	cardinal such that there exists a $K^{(k)}_{n, n}$-free, $k$-uniform hypergraph 
	$H=(\kappa, E)$ whose chromatic number is uncountable. 
	
	Set $t=kn^2+1$. 
	By a {\it delta system} we shall mean a subset $D\subseteq E$ of size $|D|=t$ 
	such that setting $d=\bigcap D$ we have $e\cap e'=d$ for all distinct $e, e'\in D$; 
	the set $d$ is also called the {\it root} of the delta system $D$. 
	For every $\ell\in [2, k)$ we define $H^{(\ell)}=(\kappa, E^{(\ell)})$ to be the 
	$\ell$-uniform hypergraph whose edges are the $\ell$-element sets which are roots 
	of such delta systems.
	
	\begin{claim}\label{clm:21}
		For every $\ell\in [2, k)$ we have $\chi(H^{(\ell)})\le \aleph_0$.
	\end{claim}    
	
	\begin{proof}
		Assume contrariwise that $\chi(H^{(\ell)})\ge \aleph_1$ holds for 
		some $\ell\in [2, k)$. Due to the minimality of $k$ the expansion 
		$K^{(\ell)}_{n, n}$ is obligatory and, hence, there is a copy 
		of $K^{(\ell)}_{n, n}$ in~$H^{(\ell )}$. Let $f_1, \dots, f_{n^2}$ be 
		the edges of such a copy. We want to extend these edges  
		recursively to edges $e_i\supseteq f_i$ of $H$ such that $e_1, \dots, e_{n^2}$ 
		form a copy of $K^{(k)}_{n, n}$. 
		
		Suppose that for some positive $i\le n^2$ the edges $e_1, \dots, e_{i-1}$
		have already been selected. Recall that $f_i$ is the root of a delta system. 
		Since the set 
		$g_i=(e_1\cup\dots\cup e_{i-1})\cup(f_i\cup\dots\cup f_{n^2})$
		satisfies $|g_i|\le kn^2<t$, this yields an edge $e_i\in E$ such 
		that $f_i\subseteq e_i$ and $g_i\cap e_i=f_i$. This completes the 
		selection of $e_1, \dots, e_{n^2}\in E$.
		
		Whenever $1\le i<j\le n^2$ we have $e_i\subseteq g_j$ and $f_j\subseteq g_i$,
		wherefore
				\[
			e_i\cap e_j
			=
			e_i\cap g_j\cap e_j
			=
			e_i\cap f_j
			=
			e_i\cap g_i\cap e_j
			=
			f_i\cap f_j\,.
		\]
				Thus the edges $e_1, \dots, e_{n^2}$ form indeed a copy 
		of $K^{(k)}_{n, n}$ in $H$, which is absurd. 
	\end{proof}
	
	Now for every $\ell\in [2, k)$ we let~$E_\ell$ be the set of all 
	edges $e\in E$ containing some $f\in E(H^{(\ell)})$,
	and then we consider the partition $E=E'\dcup E''$
	defined by $E'=\bigcup_{\ell \in [2, k)} E_\ell$ and $E''=E\sm E'$. 
	Since every proper colouring of $H^{(\ell)}$ is at the same time a proper 
	colouring of $(\kappa, E_\ell)$, Claim~\ref{clm:21} 
	yields $\chi(\kappa, E_\ell)\le \aleph_0$ for every $\ell\in [2, k)$.  
	So Fact~\ref{f:zykov} implies $\chi(\kappa, E')\le \aleph_0$ and 
	Corollary~\ref{cor:zykov} discloses $\chi(\kappa, E'')\ge \aleph_1$.
	
	Now we fix a continuous increasing sequence of elementary submodels 
	$(\ccM_i)_{1\le i<\cf(\kappa)}$ of $\ccM=\langle \kappa, E\rangle$
	with union $\ccM$ such that writing $\ccM_i=\langle M_i, E\cap M_i^{(k)}\rangle$
	we have $|M_i|<\kappa$ for every positive $i<\cf(\kappa)$. 
	Moreover, we set $M_0=\vn$ and $I_i=M_{i+1}\sm M_i$ for every $i<\cf(\kappa)$. 
		
	\begin{claim}\label{clm:22}
		For every edge $e\in E''$ there exists some $i<\cf(\kappa)$ such 
		that $|e\cap I_i|\ge k-1$.  
	\end{claim}
	
	\begin{proof}
		Let $i<\cf(\kappa)$ be maximal such that 
		$e\cap I_i\ne\vn$. Assume for the sake of contradiction 
		that the cardinality $\ell$ of $f=e\sm I_i$ is in $[2, k)$.
		Due to $e\not\in E'$ we have $f\not\in E(H^{(\ell)})$. This means that $f$ 
		fails to be the root of a delta system. 
		Let $e_0, \dots, e_{s-1}\in E\cap M_i^{(k)}$
		be a maximal sequence of edges such that $f\subseteq e_0\cap\dots\cap e_{s-1}$
		and $e_j\cap e_{j'}=f$ for all distinct $j, j'<s$. 
		
		Since $f$ is not a root, 
		we have $s<t$. Moreover, the fact that there is no edge 
		$e_\star\in E\cap M_i^{(k)}$ with $f\subseteq e_\star$ and $e_\star\cap e_j=f$
		for all $j<s$ is expressible in $\ccM_i$. 
		Owing to $\ccM_i\prec\ccM$ this statement 
		needs to hold in $\ccM$ as well. But the given edge~$e$
		contradicts this assertion. This proves $\ell\not\in [2, k)$ and our claim 
		follows.   
	\end{proof}
	
	Now for every $i<\cf(\kappa)$ we set 
		\begin{align*}
		E_i^\star&=\{e\in E''\colon e\subseteq I_i\}\,, \\
		E_i^{\star\star}&=\{e\in E''\colon |e\cap I_i|=k-1\}\,, \\
		\text{ and } \quad 
		F_i&=\{e\cap I_i\colon e\in E_i^{\star\star}\}\,.
	\end{align*}
		The minimality of $\kappa$ entails $\chi(I_i, E^\star_i)\le \aleph_0$
	for every $i<\cf(\kappa)$. As the sets $I_i$ are mutually disjoint, this 
	implies $\chi(\kappa, E^\star)\le \aleph_0$ for 
	$E^\star=\bigcup_{i<\cf(\kappa)} E_i^\star$. By Claim~\ref{clm:22} the union 
		\[
		E^{\star\star}=\bigcup_{i<\cf(\kappa)} E_i^{\star\star}
	\]
		satisfies
	$E''=E^\star\dcup E^{\star\star}$ and, therefore, Corollary~\ref{cor:zykov}
	implies $\chi(\kappa, E^{\star\star})\ge\aleph_1$. If for every $i<\cf(\kappa)$
	the chromatic number of the $(k-1)$-uniform hypergraph $(I_i, F_i)$ was at 
	most countable, then we could combine appropriate colourings $I_i\lra\omega$ 
	into a single colouring that would yield the contradiction 
	$\chi(\kappa, E^{\star\star})\le\aleph_0$.
	
	Thus we can fix some ordinal $i<\cf(\kappa)$ 
	such that $\chi(I_i, F_i)\ge\aleph_1$. Owing to our minimal choice of $k$,
	for $q=tn^3+n$ the expansion $K^{(k-1)}_{q, q}$ is obligatory. So 
	there exist two disjoint sets $X, Y\subseteq I_i$ of size $q$ such that for 
	every pair $(x, y)\in X\times Y$ there is an edge $f_{xy}\in F_i$ containing~$x$
	and~$y$. Moreover, these edges have the property that all distinct 
	pairs $(x, y), (x', y')\in X\times Y$ satisfy 
	$f_{xy}\cap f_{x'y'}=\{x, y\}\cap \{x', y'\}$. By the definition of $F_i$ 
	there is for every pair $(x, y)\in X\times Y$ an ordinal $\beta_{xy}\not\in I_i$
	such that $e_{xy}=f_{xy}\cup\{\beta_{xy}\}$ is in $E^{\star\star}_i$.
	
	\begin{figure}[h!]
	\centering
	\begin{tikzpicture}
		
		\coordinate (a12) at (0,0);
		\coordinate (a21) at (-1.1,1);
		\coordinate (a22) at (0,1);
		\coordinate (a23) at (1.1,1);
		\coordinate (a31) at (-1.1,2);
		\coordinate (a32) at (0,2);
		\coordinate (a33) at (1.1,2);
		\coordinate (a42) at (0,3.1);
		\coordinate (a51) at (-.9,3.9);
		\coordinate (a52) at (0,3.9);
		\coordinate (a53) at (.9,3.9);
		
		\draw (-5,3.9)--(2,3.9);
		\draw (-5,3.1)--(2,3.1);
		
			\qedge{(-1.1,3.9)}{(a52)}{(1.1,3.9)}{4pt}{.4pt}{black}{red, opacity=0};
				\qedge{(-4.3,3.9)}{(-3,3.9)}{(-2.4,3.9)}{4pt}{.4pt}{black}{red, opacity=0};
				\qedge{(-2,3.1)}{(-1,3.1)}{(0,3.1)}{4pt}{.4pt}{black}{red, opacity=0};
		
		\draw [amet, line width=2pt] (a12)--(a52);
		\draw [amet, line width=2pt] (a12) [out=170, in=-70] to (a21) [out=105, in = -105] to (a31)[out=60, in=220] to (a42)--(a53);
		\draw [amet, line width=2pt] (a12) [out=10, in=-110] to (a23) [out=75, in = -75] to (a33)[out=120, in=-40] to (a42)--(a51);
				
		\draw (-6,-1) rectangle (3,.5);		
		\draw (-6,5) rectangle (3,.8);		
				
		\foreach \i in {12,21,22,23,31,32,33,42,51,52,53} \fill [amet!70!black](a\i) circle (2.5pt);
		
		\node at (2.4,3.9) {$X$};
		\node at (2.4,3.05) {$Y$};
		\node at (.42,3.23) {$\upsilon$};
		\node at (-2,2.7) {$Y_\star$};
		\node at (-3.35,4.3) {$X_\star$};
		\node at (0,4.3) {$T$};
		\node at (.4,-.2) {$\beta$};
		\node at (-6.5,-.2) {\Large $B$};
		\node at (-6.5,2.8) {\Large $I_i$};
		\end{tikzpicture}
\caption{A delta system with root $\{\beta, \upsilon\}$}
\label{fig21}
\end{figure}
	
	Let us call a pair $(X_\star, Y_\star)$ of 
	subsets $X_\star\subseteq X$, $Y_\star\subseteq Y$
	{\it good} if $|X_\star|, |Y_\star|\le n$ and the $|X_\star|\cdot |Y_\star|$
	ordinals $\beta_{xy}$ with $x\in X_\star$, $y\in Y_\star$ are distinct. 
	For instance, the pair $(\vn, \vn)$ is good. Fix a good pair $(X_\star, Y_\star)$
	such that $|X_\star|+|Y_\star|$ is maximal and set 
	$B=\{\beta_{xy}\colon x\in X_\star \text{ and } y\in Y_\star\}$. 
	If $|X_\star|=|Y_\star|=n$, then the $n^2$ edges $e_{xy}$ 
	with $x\in X_\star$, $y\in Y_\star$ form a copy of $K^{(k)}_{n, n}$ in $H$, 
	which is absurd.

	Thus we can assume without loss of generality that $|X_\star|<n$. For every 
	$x\in X\sm X_\star$ the pair $(X_\star\cup\{x\}, Y_\star)$ fails to be good, 
	which means that there are ordinals $\upsilon(x)\in Y_\star$ and $\beta(x)\in B$
	such that $\beta_{x, \upsilon(x)}=\beta(x)$. There are at 
	most $|Y_\star|\cdot |B|\le n^3$ possibilities for the 
	pair $(\upsilon(x), \beta(x))$. Together with $|X\sm X_\star|\ge q-n=tn^3$
	this shows that for some pair $(\upsilon, \beta)\in Y_\star\times B$
	there exists a set $T\subseteq X\sm X_\star$ of size $t$ such that 
	$(\upsilon(x), \beta(x))=(\upsilon, \beta)$ holds for every $x\in T$.
	The $t$ edges $e_{x\upsilon}$ with $x\in T$ form a delta system with 
	root $\{\upsilon, \beta\}$. But this contradicts the fact that these 
	edges are in $E''$. This proves Theorem~\ref{thm:12}.
\end{proof}

\subsection*{Acknowledgement} We are grateful to {\sc Joanna Polcyn} for the 
beautiful figures.

\end{document}